\documentclass{amsart}

\usepackage{amssymb}
\usepackage{amsmath}
\usepackage{amsthm}

\newtheorem {theorem} {Theorem}[section]
\newtheorem {proposition} [theorem]{Proposition}
\newtheorem {corollary} [theorem]{Corollary}
\newtheorem {lemma}  [theorem]{Lemma}
\numberwithin{equation}{section}
\theoremstyle{definition}
\newtheorem{df}{Definition}

\theoremstyle{remark}
\newtheorem{rem}{Remark}

\newcommand{\al}{\alpha}

\newcommand{\clk}{\sigma}

\newcommand{\plh}{PM}

\newcommand{\diin}{\theta}

\newcommand{\za}{\zeta}

\newcommand{\ph}{\varphi}

\newcommand{\hol}{\mathcal{H}ol}
\newcommand{\Rl}{\mathrm{Re\,}}

\newcommand{\Dbb}{\mathbb D}
\newcommand{\Tbb}{\mathbb T}
\newcommand{\dn}{{\mathbb D}^n}
\newcommand{\tn}{{\mathbb T}^n}
\newcommand{\mn}{m_n}
\newcommand{\Cbb}{\mathbb C}

\newcommand{\Zbb}{\mathbb Z}
\newcommand{\kla}{I^*(H^2)}

\begin{document}

\title[Clark measures]{Clark measures on the torus}

\author{Evgueni Doubtsov}

\address{St.~Petersburg Department
of Steklov Mathematical Institute, Fontanka 27, St.~Petersburg 191023, Russia and}

\address{
Department of Mathematics and Computer Science,
St.~Petersburg State University,
Line 14th (Vasilyevsky Island), 29, St.~Petersburg 199178,
Russia}

\email{dubtsov@pdmi.ras.ru}

\thanks{This research was supported by the Russian Science Foundation (grant No.~19-11-00058).}

\subjclass[2010]{Primary 30J05, 32A35; Secondary 31C10, 46E27, 46J15}

\keywords{Hardy space, polydisc, inner function, model space, pluriharmonic measure}

\begin{abstract}
Let $\mathbb{D}$ denote the unit disc of $\mathbb{C}$ and let $\mathbb{T}= \partial\mathbb{D}$.
Given a holomorphic function $\varphi: \mathbb{D}^n \to \mathbb{D}$, $n\ge 2$, we study the corresponding family
$\sigma_\alpha[\varphi]$, $\alpha\in\mathbb{T}$,
of Clark measures on the torus $\mathbb{T}^n$.
If $\varphi$ is an inner function, then we introduce
and investigate related isometric operators $T_\alpha$
mapping analogs of model spaces into $L^2(\sigma_\alpha)$, $\alpha\in\mathbb{T}$.
\end{abstract}

\maketitle

\section{Introduction}\label{s_int}

Let $\Dbb$ denote the open unit disc of $\Cbb$ and let $\Tbb=\partial\Dbb$.
For $n=1,2, \dots$, the equality
\[
C(z, \za) = \prod_{j=1}^{n}\frac{1}{1- z_j \overline{\za}_j}, \quad z\in\dn,\ \za\in\tn,
\]
defines the Cauchy kernel for $\dn$.
The corresponding Poisson kernel is given by the formula
\[
P(z, \za) = \frac{C(z, \za) C(\za, z)}{C(z, z)},
\quad z\in\dn,\ \za\in\tn.
\]
Let $M(\tn)$ denote the space of complex Borel measures on the torus $\tn$.
For $\mu\in M(\tn)$, the Cauchy transform $\mu_+$ is defined as
\[
  \mu_+(z) = \int_{\tn} C(z, \xi)\, d\mu(\xi), \quad z\in \dn.
\]

\subsection{Pluriharmonic measures}
A measure $\mu\in M(\tn)$, $n\ge 2$, is called \textsl{pluriharmonic} if the Poisson integral
\[
P[\mu](z) = \int_{\tn} P(z, \za)\, d\mu(\za), \quad z\in\dn,
\]
is a pluriharmonic function.
Let $\plh(\tn)$ denote the set of all pluriharmonic measures.
Observe that $\mu\in\plh(\tn)$
if and only if the Fourier coefficients of $\mu$
are equal to zero outside the set $(-\Zbb_+^n) \cup \Zbb_+^n$.

\subsection{Clark measures}
Given an $\alpha\in\Tbb$ and a holomorphic function $\ph: \dn\to \Dbb$, the quotient
\[
\frac{1-|\ph(z)|^2}{|\al-\ph(z)|^2}= \Rl \left(\frac{\al+ \ph(z)}{\al- \ph(z)} \right), \quad z\in \dn,
\]
is positive and pluriharmonic.
Therefore, there exists a unique positive measure $\clk_\al= \clk_\al[\ph] \in M(\tn)$
such that
\[
P[\clk_\al](z) = \Rl \left(\frac{\al+ \ph(z)}{\al- \ph(z)} \right), \quad z\in \dn.
\]
By the definition of $\plh(\tn)$, we have $\clk_\al\in \plh(\tn)$.

After the seminal paper of Clark \cite{Cl72},
various properties and applications of the measures $\clk_\al$ on the unit circle $\Tbb$ have been obtained;
see, for example, reviews \cite{MM06, PS06, GR15} for further references.
To the best of the author's knowledge, the measures $\clk_\al$ on $\tn$, $n\ge 2$, have not been
investigated earlier.
In particular, in the present paper, we show that the properties of $\clk_\al$ on $\tn$ are quite different from
those of $\clk_\al$
on the unit circle $\Tbb$ and
on the unit sphere $S_n$ of $\mathbb{C}^n$, $n\ge 2$
(cf.\ \cite{ADcras}).

\subsection{Clark measures and model spaces}
Let $\mn$ denote the normalized Le\-besgue measure on $\tn$.

\begin{df}\label{d_inner}
A holomorphic function $I:\dn \to \Dbb$ is called \textsl{inner}
if $|I(\za)|=1$ for $\mn$-a.e. $\za\in\tn$.
\end{df}

In the above definition,
$I(\za)$ stands, as usual, for
$\lim_{r\to 1-} I(r\za)$.
Recall that the corresponding limit exists $\mn$-a.e.
Also, by the above definition, every inner function is non-constant.
The present paper is primarily motivated by the studies of Clark \cite{Cl72}
related to the measures $\clk_\al [\ph] \in M(\Tbb)$, $\al\in\Tbb$,
where $\ph$ is an inner function in $\Dbb$.

Given an inner function $I$ in $\dn$, we have
\[
P[\clk_\al](\za) =\frac{1-|I(\za)|^2}{|\al-I(\za)|^2}=0 \quad \mn\textrm{-a.e.},
\]
therefore, $\clk_\al = \clk_\al[I]$ is a singular measure. Here and in what follows, this means that
$\clk_\al$ and $\mn$ are mutually singular; in brief, $\clk_\al \bot\mn$.
Standard properties of Poisson integrals guarantee that $\clk_\al$, $\al\in\Tbb$, is supported
by the set
$E_\al = \{\za\in \tn: I(\za)=\al\}$.
In particular, $\clk_\al \bot \clk_\beta$ for $\al\neq \beta$.

For $n\ge 1$, let $\hol(\dn)$ denote the space of holomorphic functions in $\dn$.
For $0<p<\infty$, the classical Hardy space $H^p=H^p(\dn)$ consists of those $f\in \hol(\dn)$ for which
\[
\|f\|_{H^p}^p = \sup_{0<r<1} \int_{\tn} |f(r\za)|^p\, d\mn(\za) < \infty.
\]
As usual, we identify the Hardy space $H^p(\dn)$, $p>0$, and the space
$H^p(\tn)$ of the corresponding boundary values.

For an inner function $\diin$ on $\Dbb$, the classical
model space $K_\diin$ is defined as $K_\diin = H^2(\Tbb)\ominus \diin H^2(\Tbb)$.
Clark \cite{Cl72} introduced and studied a family of unitary operators
$U_\al : K_\diin \to L^2(\clk_\al)$, $\al\in\Tbb$.

For an inner function $I$ in $\dn$, $n\ge 2$, there are several analogs of $K_\diin$.
We consider the following direct analog of $K_\diin$:
\[
  \kla = H^2 \ominus I H^2.
\]
In this paper, we define isometric operators
\[
T_\al: \kla\to L^2(\clk_\al), \quad \al\in\Tbb,
\]
and we address the problem whether $T_\al$ is a unitary operator for a given inner function $I$.
Also, we give examples of inner functions with different sets
$\{\al\in\Tbb: T_\al\ \textrm{extends to a unitary operator}\}$.

\subsection*{Organization of the paper}
Auxiliary properties of Clark measures are obtained in Section~\ref{s_aux_clk}.
In particular, we prove two decomposition results: each pluriharmonic measure decomposes in terms of its slices,
and Lebesgue measure $\mn$ disintegrates in terms of Clark measures.
The isometries $T_\al: \kla\to L^2(\clk_\al)$ are introduces and studied in Section~\ref{s_clk_inner}.
Explicit examples of Clark measures are given in the final Section~\ref{s_ex}.

\subsection*{Acknowledgement}
The author is grateful to Aleksei B.\ Aleksandrov
for helpful suggestions and remarks.

\section{Basic properties of Clark measures}\label{s_aux_clk}

The following lemma is standard, so we omit its proof.

\begin{lemma}\label{l_MS_wlim}
Let $\{\mu_k\}_{k=1}^\infty$ be a bounded sequence
in $M(\tn)$ and let $\mu\in M(\tn)$. Then $\lim\limits_{k\to\infty}\mu_k=\mu$
in $\clk(M(\tn), C(\tn))$-topology if and only if
$\lim\limits_{k\to\infty} {P}[\mu_k](z)= {P}[\mu](z)$ for all $z\in\dn$.
\end{lemma}

\begin{corollary}\label{c_slices_wcont}
Let $\ph: \dn\to \Dbb$ be a holomorphic function. The
mapping $\al\mapsto\clk_\al[\ph]$
is continuous from $\Tbb$ into the space $M(\tn)$ endowed with the weak topology.
\end{corollary}

Next, we obtain a result related to all pluriharmonic measures.

\subsection{A decomposition theorem for pluriharmonic measures}
It is well known that
\begin{equation}\label{e_Leb_dcm}
\int_{\tn} f\, d\mn = \int_{\tn} \int_{\tn} f \, dm_1^\za\, d\mn(\za), \quad f\in C(\tn),
\end{equation}
where $m_1^\za \in M(\tn)$, $\za\in\tn$, is the normalized one-dimensional Lebesgue measure
supported by the unit circle $\Tbb\za \subset \tn$
and considered as a measure on $\tn$.

Now, let $\mu$ be a pluriharmonic measure on $\tn$. Put $u=P[\mu]$
and $u_r(\za) = u(r\za)$, $0\le r <1$, $\za\in\tn$.
For $\za\in\tn$, the slice function $u_\za$ is defined as $u_\za(\lambda) = u(\lambda\za)$, $\lambda\in\Dbb$.
Since $u$ is pluriharmonic, $u_\za$ is harmonic for all $\za\in\tn$.
Also, by the monotone convergence theorem
\begin{equation}\label{e_u_xi}
\begin{aligned}
\int_{\tn}\sup_{0<r<1}\|(u_\za)_r\|_{L^1(\Tbb)}\, d\mn(\za)
&= \int_{\tn}\lim_{r\to 1-}\|(u_\za)_r\|_{L^1(\Tbb)}\, d\mn(\za) \\
&= \lim_{r\to 1-} \|u_r\|_{L^1(\tn)} < \infty.
\end{aligned}
\end{equation}
Hence, for $\mn$-a.e.\ $\za\in\tn$, we have
\[
\sup_{0<r<1} \|(u_\za)_r\|_{L^1(\Tbb)} <\infty.
\]
Therefore, for $\mn$-a.e.\ $\za\in\tn$,
there exists $\mu_\za\in M(\tn)$ such that
$\textrm{supp\,}\mu_\za \subset \Tbb\za$ and
\[
u_{\za}(\lambda) = \int_{\tn} \frac{1-|\lambda|^2}{|w- \lambda\za|^2}\, d\mu_\za(w), \quad \lambda\in\Dbb.
\]
If $\mu$ is a positive pluriharmonic measure,
in particular, a Clark measure, then the slice measure $\mu_\za$ is clearly defined for any $\za\in\tn$.

Using \eqref{e_u_xi}, we also conclude that
\begin{equation}\label{e_muz_norm}
\int_{\tn} \|\mu_\za\|\, d\mn(\za) < \infty
\end{equation}
for any $\mu\in \plh (\tn)$.

\begin{proposition}\label{p_slices}
  Let $\mu\in\plh(\tn)$ and let $\mu_\za$ denote the slice measure defined as above for $\mn$-a.e.\ $\za\in\tn$.
  Then
  \[
  \mu = \int_{\tn} \mu_\za\, d\mn(\za)
  \]
in the following weak sense:
\[
  \int_{\tn} f d\mu = \int_{\tn} \int_{\tn} f d\mu_\za\, d\mn(\za)
  \]
for all $f\in C(\tn)$.
\end{proposition}
\begin{proof}
Let $f\in C(\tn)$.
Given a measure $\mu\in \plh(\tn)$,
put $u = P[\mu]$.
By \eqref{e_Leb_dcm},
\[
\int_{\tn} f u_r\, d\mn = \int_{\tn} \int_{\tn} f u_r\, dm_1^\za\, d\mn(\za), \quad 0<r<1.
\]
Observe that $u_r \mn \to \mu$ and $u_r m_1^\za\to \mu_\za$
as $r\to 1-$
in $\sigma(M(\tn), C(\tn))$-topology for $\mn$-a.e.\ $\za\in\tn$.
So, taking the limit as $r\to 1-$
and applying \eqref{e_muz_norm}, we obtain
\[
\int_{\tn} f\, d\mu = \int_{\tn} \int_{\tn} f \, d\mu_\za\, d\mn(\za)
\]
by the dominated convergence theorem.
\end{proof}

\subsection{A disintegration theorem for Clark measures}
Proposition~\ref{p_slices}
indicates that the Clark measures on $\tn$ could inherit various properties of
the classical Clark measures on the unit circle.
As an illustration, we prove the following analog of the disintegration theorem
obtained in \cite{Aab87} for $d=1$.

\begin{proposition}\label{t_desint}
Let $\ph: \dn\to\Dbb$ be a holomorphic function and
let $\clk_\al = \clk_\al[\ph]$, $\al\in \Tbb$.
Then
\[
\int_{\Tbb}
\int_{\tn}  f\, d\clk_\al\, d m_1(\al) = \int_{\tn} f\, d\mn
\]
for all $f\in C(\tn)$.
\end{proposition}

\begin{proof}
Observe that the slice measures $(\clk_\al)_\za$ are defined for all $\al\in\Tbb$, $\za\in\tn$.
The norms of $(\clk_\al)_\za$, $\al\in\Tbb$, $\za\in\tn$, are bounded by a constant $C(I)>0$.
Also, $(\clk_\al)_\za[I]$ coincides with $(\clk_\al)[I_\za]$ and
\begin{equation}\label{e_clk_slice}
\int_{\Tbb}\int_{\tn}
f\, d\clk_\al[I_\za]\, dm_1(\al)
= \int_{\tn} f\, dm_1^\za, \quad f\in C(\tn),
\end{equation}
by the one-dimensional disintegration theorem (see \cite{Aab87}).
So, applying Proposition~\ref{p_slices} Fubini's theorem and \eqref{e_Leb_dcm}, we obtain
\[
\begin{aligned}
  \int_{\Tbb} \int_{\tn} f \, d\clk_\al\, d m_1(\al)
  &= \int_{\Tbb} \int_{\tn} \int_{\tn} f \, d(\clk_\al)_{\za}\,d\mn(\za)\, d m_1(\al)  \\
  &= \int_{\tn}\int_{\Tbb}\int_{\tn} f \, d\clk_\al[I_\za]\, d m_1(\al)\,d\mn(\za) \\
  &= \int_{\tn}\int_{\tn}  f \, d m_1^{\za}\,d\mn(\za) \\
  &= \int_{\tn} f \, d\mn(\za),
\end{aligned}
\]
as required.
\end{proof}

\subsection{Clark measures and Cauchy kernels}
The following lemma is a particular case of Exercise~1
from \cite[Chap.~8]{KrBook}.

\begin{lemma}\label{l_diag}
Let $F$ be a holomorphic function on $\dn\times \dn$.
If $F(z, \overline{z}) =0$ for all $z\in\dn$,
then $F(z, w) =0$ for all $(z, w) \in \dn\times \dn$.
\end{lemma}

\begin{proposition}\label{p_cauchy_dbl}
Let $\ph: \dn\to\Dbb$, $d\ge 2$, be a holomorphic function and
let $\clk_\al = \clk_\al[\ph]$, $\al\in \Tbb$.
 Then
  \[
  \int_{\tn} C(z, \za) C(\za, w)\, d\clk_\al(\za) =
  \frac{1- \ph(z)\overline{\ph(w)}}{(1-\overline{\al}{\ph(z)})(1-\al\overline{\ph(w)})} C(z,w)
  \]
for all $\al\in\Tbb$, $z, w \in\dn$.
\end{proposition}
\begin{proof}
  The equality
  \[
  \int_{\tn} P(z, \za) \, d\clk_\al(\za) = \frac{1-|\ph(z)|^2}{|\al- \ph(z)|^2}, \quad z\in \dn,
  \]
  and the definition of $P(z,\za)$ provide
   \[
  \int_{\tn} C(z, \za) C (\za, z)\, d\clk_\al(\za) = \frac{1-|\ph(z)|^2}{|\al- \ph(z)|^2} C(z,z), \quad z\in \dn.
  \]
  It remains to apply Lemma~\ref{l_diag}.
\end{proof}

\begin{corollary}\label{c_cauchyof_clk}
Let $\ph: \dn\to\Dbb$, $d\ge 2$, be a holomorphic function
and let $\clk_\al = \clk_\al[\ph]$, $\al\in \Tbb$.
Then
\[
(\clk_\al)_+(z) = \frac{1}{1-\overline{\al} \ph(z)} + \frac{\al\overline{\ph(0)}}{1-\al\overline{\ph(0)}}
\]
for all $\al\in\Tbb$, $z\in\dn$.
\end{corollary}
\begin{proof}
  Apply Proposition~\ref{p_cauchy_dbl} with $w=0$.
\end{proof}

\section{Clark measures for inner functions}\label{s_clk_inner}
In this section, we assume that $\ph$ is an inner function.
So, we use the symbol $I$ in the place of $\ph$.

\subsection{Abstract approach to isometries $T_\al: \kla \to L^2(\clk_\al)$}
Let $I$ be an inner function in $\dn$.
Given $f, g\in \kla$, we claim that
\begin{equation}\label{e_L2ort}
\int_{\tn}f \overline{g} \overline{I}^k\, d\mn =0\quad\textrm{for all}\ k\in \Zbb\setminus \{0\}.
\end{equation}
Indeed, without loss of generality, we may assume that $k\ge 1$.
Observe that $ f \overline{g} \bot I^k$ if and only if $f \bot I^k g$.
So, it suffices to apply the definition of $\kla$.

Now assume that $f, g\in \kla\cap C(\tn)$.
Recall that $I = \al$ $\clk_\al$-a.e.
Hence, combining \eqref{e_L2ort} and Proposition~\ref{t_desint}, we obtain
\[
0 = \int_{\Tbb}
\int_{\tn}  f\overline{g} \overline{I}^k\, d\clk_\al\, d m_1(\al)
= \int_{\Tbb} \overline{\al}^k
\int_{\tn}  f\overline{g}\, d\clk_\al\, d m_1(\al) \quad \textrm{for all}\ k\in \Zbb\setminus \{0\}.
\]
Therefore,
\[
\int_{\tn}  f\overline{g}\, d\clk_\al =\textrm{const for\ } m_1\textrm{-a.e.\ } \al\in\Tbb.
\]
In fact, applying Corollary~\ref{c_slices_wcont}, we conclude that the above property holds for all $\al\in\Tbb$.
So, we have
\[
\int_{\tn} f\overline{g}\, d\mn = \int_{\tn}  f\overline{g}\, d\clk_\al
\]
for all $\al\in\Tbb$.
Thus, given an $\al\in\Tbb$, we have an isometry $T_\al: \kla\to L^2 (\clk_\al)$.

\subsection{Constructive approach to isometries $T_\al: \kla \to L^2(\clk_\al)$}
We have
\[
C(\za, z) = (1- \overline{I(z)} I(\za))C(\za, z)
+ \overline{I(z)} I(\za) C(\za, z)
\in \kla\oplus I H^2
\]
as functions of $\za$. In other words,
\[
K(z, \za) = \frac{1-I(z)\overline{I(\za)}}{\prod_{j=1}^n (1-z_j\overline{\za_j})}
= (1-I(z)\overline{I(\za)}) C(z, \za)
\]
is the reproducing kernel for $\kla$.

Put
$K_w(z) = K(z, w)$ and define
\[
(T_\al K_w)(\xi) = \frac{1-\al \overline{I(w)}}{\prod_{j=1}^n (1-\xi_j\overline{w_j})}
= (1-\al \overline{I(w)}) C(\xi, w), \quad \xi\in\tn.
\]

\begin{theorem}\label{t_kla}
For each $\al\in\Tbb$, $T_\al$ has a unique extension
to an isometric operator from $\kla$ into $L^2(\clk_\al)$.
\end{theorem}
\begin{proof}
Fix an $\al\in\Tbb$.
%Now, we claim that $(U_\al K_w, U_\al K_z)_{L^2(\clk_\al)}=(K_w, K_z)_{H^2}$ for $z, w \in \bd$.
Applying Proposition~\ref{p_cauchy_dbl}, we obtain
  \begin{align*}
    (T_\al K_w, T_\al K_z)_{L^2(\clk_\al)}
    &=\int_{\tn} (1-\al \overline{I(w)}) C(\za, w) (1-\overline{\al} I(z)) C(z, \za)\, d\clk_\al(\za) \\
    &=(1-\al \overline{I(w)})(1-\overline{\al} I(z)) \int_{\tn} C(\za, w) C(z, \za)\, d\clk_\al(\za) \\
    &= (1- I(z) \overline{I(w)}) C(z, w)\\
    &= K(z,w) = (K_w, K_z)_{H^2}.
  \end{align*}
So, $T_\al$ extends to an isometric embedding of $\kla$ into $L^2(\clk_\al)$.
Since the linear span of the family $\{K_w\}_{w\in\dn}$ is dense in $\kla$,
the extension is unique.
\end{proof}

By definition, the polydisc algebra $A(\dn)$ consists of those $f\in C(\overline{\dn})$ which are holomorphic in $\dn$.

\begin{theorem}\label{t_tal_on}
Let $I$ be an inner function in $\dn$, $n\ge 2$, $\al\in\Tbb$.
Then the following properties are equivalent:
\begin{enumerate}
  \item[(i)] $T_\alpha$ is a unitary operator;
  \item[(ii)] $(f\clk_\al)_+ \not\equiv 0$ for any $f\in L^2(\clk_\al)$, $f\not\equiv 0$;
  \item[(iii)] the polydisc algebra $A(\dn)$ is dense in $L^2(\clk_\al)$.
\end{enumerate}
\end{theorem}
\begin{proof}
(i)$\Rightarrow$(ii)
Let $f\in \kla$ and $\al\in\Tbb$.
Since $K(z, \cdot) \in \overline{\kla}$ and $T_\al : \kla\to L^2(\clk_\al)$ is unitary,
we obtain
\[
\begin{aligned}
f(z)
&= \int_{\tn} f(\za) K(z, \za)\, d\mn(\za) \\
&= \int_{\tn} f(\za) K(z, \za)\, d\clk_\al(\za) \\
&= \int_{\tn} (1- \overline{\al} I(z)) f(\za) C(z, \za)\, d\clk_\al(\za) \\
&= (1- \overline{\al} I(z)) (f\clk_\al)_+(z).
\end{aligned}
\]
Now, assume that $(f\clk_\al)_+ \equiv 0$. Then $f\equiv 0$ as an element of $\kla$ or,
equivalently, as an element of $ L^2(\clk_\al)$.
So (i) implies (ii).

(ii)$\Rightarrow$(iii)
Assume that (iii) does not hold.
Then there exists $f\in L^2(\clk_\al)$, $f\not\equiv 0$, such that
\[
\int_{\tn} f \overline{h}\, d\clk_\al =0
\]
for all $h\in A(\dn)$, in particular, for $h(\za)= C(\za, z)$, $z\in \dn$.
So $(f\clk_\al)_+ \equiv 0$ and we arrive to a contradiction.

(iii)$\Rightarrow$(i)
If (iii) holds, then the family
\[
\{(1-\al \overline{I(w)}) C(\xi, w)\}_{w\in \dn}
\]
is dense in $L^2(\clk_\al)$.
So, we continue the proof of Theorem~\ref{t_kla}
and conclude that $T_\al$ is onto,
that is, $T_\al$ extends to a unitary operator.
\end{proof}

\begin{rem}
Theorem~\ref{t_tal_on} also holds for the inner functions in the unit ball $B_n$ of $\Cbb^n$, $n\ge 1$.
However, it degenerates: the ball algebra $A(B_n)$ is dense in $L^2(\clk_\al[I])$
for any inner function $I$ and any $\al\in\Tbb$; see \cite{ADcras}.
This is not the case for various inner functions in $\dn$; see Section~\ref{s_ex}.
\end{rem}

\section{Examples}\label{s_ex}

In this section, we use the symbol $m$ in the place of $m_1$ to denote the normalized Lebesgue measure on $\Tbb$.

\subsection{All $T_\al$ are not unitary}
Let $I(z) =z_1$, $z\in\dn$, $n\ge 2$. For $\al\in\Tbb$,
the measure $\clk_\al= \clk_\al[I]$, $\al\in\Tbb$, is supported by the set
$K_\al = \{\za\in\tn: I(\za) =\al\} = \{\za\in\tn: \za_1 =\al\}$.
Also, we have
\[
\clk_\al = \delta_\al(\za_1) \otimes m(\za_2) \otimes\dots \otimes m(\za_n).
\]
Clearly, the polydisc algebra $A(\dn)$ is not dense in $L^2(\clk_\al)$,
so $T_\al$ is not unitary for all $\al\in\Tbb$.

\subsection{All $T_\al$ are unitary}
Let $I(z)=z_1 z_2$, $z\in\Dbb^2$.
For $\al\in\Tbb$, the measure $\clk_\al= \clk_\al[I]$
is supported by $K_\al = \{(\xi, \al\overline{\xi}): \xi\in\Tbb\}$.
Now, let $P_z(\za)$ denote the Poisson kernel $P(z, \za)$, $z\in \Dbb^2$, $\za\in \Tbb^2$.
On the one hand, the Poisson integral formula in dimension one and the definition of $\clk_\al$
guarantee that
\[
 \int_{\Tbb} \frac{1-|z_1^2|}{|z_1 - \xi|^2} \frac{1- |z_2\xi|^2}{|z_2 \xi - \al|^2}\, d m_1(\xi)
= \frac{1-|z_1 z_2|^2}{|\al -z_1 z_2|^2}
= \int_{\Tbb^2} P_z(\za)\, d\clk_\al(\za).
\]
On the other hand, the same integral is equal to
\[
 \int_{\Tbb} \frac{1-|z_1^2|}{|z_1 - \xi|^2} \frac{1- |z_2|^2}{|z_2 - \al\overline{\xi}|^2}\, d m_1(\xi)
= \int_{\Tbb} P_z(\xi, \al\overline{\xi})\, d m_1(\xi).
\]
Therefore,
\[
\int_{\Tbb^2} f(\za)\, d\clk_\al(\za)  =  \int_{\Tbb} f(\xi, \al\overline{\xi})\, d m_1(\xi)
\]
for all $f\in C(\Tbb^2)$.
The functions $f(\xi, \al\overline{\xi})$, $f$ is a holomorphic monomial,
are dense in $L^2(m_1)$, hence, $A(\Dbb^2)$ is dense in $L^2(\clk_\al)$.
So, $T_\al$ is a unitary operator for any $\al\in \Tbb$.

\subsection{All $T_\al$ are unitary operators except $T_{-1}$}
Put
\[
I(z) = \frac{z_1 + z_2 + 2z_1 z_2}{z_1 + z_2 +2}, \quad z\in\Dbb^2.
\]
Observe that $\left|I(z) \overline{z_1} \overline{z_2}\right| =1$, hence, $I$ is an inner function.
For $\al\in\Tbb$, $\al\neq -1$, direct computations show that
the measure $\clk_\al= \clk_\al[I]$ is supported by the set
\[
K_\al = \{(\xi, \overline{b_\al(\xi))}: \xi\in\Tbb\},
\]
where
\[
b_\al(\xi) = \frac{1+ 2\xi -\al}{(\al-1) \xi + 2\al}, \quad \xi\in\Tbb.
\]
Also, the density of $\clk_\al$ with respect to $\xi$
has a zero for $\xi=-1$, that is, at the point $(-1, -1)\in\Tbb^2$.

If $\al\neq -1$, then $|b_\al(\xi)|=1$ for all $\xi\in\Tbb$.
So, $b_\al$ is a Blaschke factor; in particular, $b_1(\xi)=\xi$, $\xi\in\Tbb$.
Hence, $f(\xi, \overline{b_\al}(\xi))$, $f\in A(\Dbb^2)$, is dense in $L^2(\Tbb)$.
In other words,
$A(\Dbb^2)$ is dense in $L^2(\clk_{\al})$ and $T_\al$ is a unitary operator for $\al\neq -1$.

The point $\al=-1$ is exceptional:
\[
\aligned
\Rl\left( \frac{\al +I}{\al -I} \right) &= \Rl\left( \frac{1-I}{1+I} \right)
 = \Rl\left( \frac{1- z_1 z_2}{(z_1 +1)(z_2 +1)} \right) \\
&=  \Rl\left( \frac{1}{z_1 +1} -\frac{1}{2} + \frac{1}{z_2 +1} - \frac{1}{2} \right)
= \frac{1}{2}\left( \frac{1- |z_1|^2}{|z_1+1|^2} + \frac{1- |z_2|^2}{|z_2+1|^2} \right).
\endaligned
\]
Hence, $K_{-1} = \{(\xi, -1): \xi\in\Tbb\} \cup \{(-1, \xi): \xi\in\Tbb\}$ and
\[
\clk_{-1} = \delta_{-1} \otimes m(z_2) + m(z_1) \otimes \delta_{-1}.
\]
Clearly, the bidisc algebra $A(\Dbb^2)$ is not dense in $L^2(\clk_{-1})$,
so the operator $T_{-1}$ is not unitary.

\end{document}